\documentclass[preprint,12pt]{elsarticle}

\usepackage{amssymb}

\usepackage{amsmath}
\usepackage{amssymb}
\usepackage{amsfonts}
\usepackage{latexsym}
\usepackage{color}
\usepackage{enumerate}
\usepackage{lineno}

\usepackage{tikz}
\usepackage{pgfplots}
\usepackage{tikz-3dplot}
\tdplotsetmaincoords{60}{115}
\pgfplotsset{compat=newest}

\catcode `\@=11 \@addtoreset{equation}{section}

\catcode `\@=12
  \newcommand{\mc}{\mathcal}
\newcommand{\F}{{\cal F}}
\newcommand{\K}{\mc K}
\newcommand{\Hil}{\mc H}
\newcommand{\M}{\mc M}
\newtheorem{thm}{Theorem}
\newtheorem{cor}[thm]{Corollary}
\newtheorem{lemma}[thm]{Lemma}
\newtheorem{prop}[thm]{Proposition}

\newtheorem{example}[thm]{Example}
\newtheorem{rem}[thm]{Remark}
\newenvironment{proof}{\noindent {\bf Proof --}}{\hfill$\square$ \vspace{3mm}\endtrivlist}
\numberwithin{equation}{section}

\catcode `\@=11 \@addtoreset{equation}{section}
\catcode `\@=12
\begin{document}

\begin{frontmatter}



\title{On Description of Dual Frames}


\author[1]{Alan Kamuda}
\ead{kamuda@agh.edu.pl}

\author[1]{Sergiusz Ku\.zel \corref{cor1}}
\ead{kuzhel@agh.edu.pl}

\cortext[cor1]{Corresponding author}

\address[1]{AGH University of Science and Technology, 30-059 al. Mickiewicza 30, Krak\'{o}w, Poland}

\begin{abstract}
One of a key problems in signal reconstruction process with the use of frames is to find a dual frame. 
Typically, a canonical dual frame is used. However, there are many applications where this choice appears to be unfortunate. 
Due to that fact, it is necessary to develop a tool, which helps to find a suitable dual frame. In this paper we give a method to find every dual frames. 
The proposed method is based on Naimark's dilation theorem  and the obtained description of dual frames involves parameters that characterize
extension of a Parseval frame to an orthonormal basis.  These formulas are simplified for frames in finite-dimensional spaces and for near-Riesz bases. 
In the latter case, the simplification is based on the extended and supplemented version
of the Naimark theorem, which is proved in the last part of the paper.  
\end{abstract}


\end{frontmatter}






\section{Introduction}
 Frames attract a steady interests in recent  research in applied mathematics because 
  they are used in various areas such as operator theory, signal processing, computer science, engineering, quantum information theory, see, e.g.,  
  \cite{BCEK, heil, KK, MHC, S} and references therein. They are unavoidable due to the flexibility of  reconstruction. 
  This flexibility is  achievable by many different reconstruction formulas for recovering a signal, which is not the case for bases. 
  One possible decomposition might be obtained with the use of canonical dual frame. 
  However, for some applications, other (alternate) dual frames might be of interest. For example, when frames with 
  specific structure are used, then it might be useful to use representations based on
a dual frame with the same structure. In \cite{DHnew} one can find an example of a wavelet frame for 
which canonical dual is not a wavelet frame, but there are infinitely many other
dual frames having the wavelet-structure. Another motivation for the consideration of
dual frames other then the canonical dual deals with signal restoration and noise reduction
problems \cite[Section 9.3]{LO}.
Alternate dual frames might provide a high precision
linear reconstruction procedure for sigma-delta quantization \cite{BLPY, LPY1} and they 
seem to be helpful in sparse modeling of signals \cite[Section 3.4]{BDE}.
For this reason, it becomes important to obtain a variety of formulas describing all dual frames for a given frame $\F_\varphi=\{\varphi_j, \,j\in \mathbb{J}\}$
  in a Hilbert space $\K$. 
  
 We recall that a frame $\F_\psi=\{\psi_j,  j\in \mathbb{J}\}$ satisfying the condition
$$
f=\sum_{j\in\mathbb{J}}(f, \psi_j)\varphi_j=\sum_{j\in\mathbb{J}}(f, \varphi_j)\psi_j, \qquad f\in\K
$$
is called a \emph{dual frame} of $\F_\varphi$.  An example of dual frame of $\F_\varphi$ is  its \emph{canonical dual} 
$\F_\psi=S^{-1}\F_\varphi$, where $S$ is the frame operator of $\F_\varphi$. 
A dual frame which is not canonical is called \emph{alternate}. 
  
All dual frames $\F_\psi$ of $\F_\varphi$  might be described with the use of Bessel sequences $\{h_j, j\in \mathbb{J}\}$.
  Namely \cite[Theorem 6.3.7]{chri}, \cite{SL}, 
  \begin{equation}\label{1963new}
  \psi_j=S^{-1}\varphi_j+h_j-\sum_{i\in\mathbb J}(S^{-1}\varphi_j,\varphi_i)h_i,
  \end{equation}
  where $S$ is the frame operator of $\F_\varphi$.  The canonical dual frame corresponds to the zero sequence $\{h_i\}$.  However, the relationship between a dual frame $\F_\psi$ and a Bessel sequence $\{h_i\}$ in \eqref{1963new} is not always satisfactory because it does not ensure the selection of dual frames with prescribed properties  (potential, frame operator, optimal frame bounds, etc). 

In the present paper we develop another approach to the description of dual frames
that allows one to express properties of dual frames in terms of parameters of description.
These parameters are chosen with the use of well known fact\footnote{to the best of our knowledge, it was firstly noticed by Heil in \cite[Corollary 6.3.5]{Heil1}} that 
each frame $\F_\varphi=\{\varphi_j, \,j\in \mathbb{J} \}$ admits the presentation 
 $\F_\varphi=S^{1/2}\F_e$,  where $S$ is the frame operator of $\F_\varphi$ and $\F_e=\{e_j, \,j\in \mathbb{J}  \}$ is a Parseval frame  (PF). For technical reasons, we rewrite the frame operator as $S=e^Q$, where $Q$ is a bounded self-adjoint operator.
Then $\F_\varphi=e^{Q/2}\F_e$. It is principal for our approach that the operator $Q$ and the   
PF $\F_e$ \emph{are determined uniquely} by the given frame  $\F_\varphi$ (see Remark \ref{AGHWMS16}).
This means that the collection of all frames in $\K$ is in  
one-to-one correspondence with the set of pairs $\{Q, \F_e\}$, where $Q$ is a bounded self-adjoint operator in $\K$ and $\F_e$ is a PF in $\K$. This relationship is sufficiently informative since $Q$ determines the frame operator, the frame bounds, and the potential
of $\F_\varphi$, while the excess of $\F_\varphi$ is characterized by $\F_e$ (Section \ref{sec2.1}). 

Given a frame $\F_\varphi=e^{Q_\varphi/2}\F_e$, its duals
$\F_\psi$ are characterized by certain pairs $\{Q, \F_e'\}$, i.e.
$\F_\psi=e^{Q/2}\F_e'$. \emph{Could one specify these parameters?} 
We completely solved this problem in Section \ref{sec2.2} assuming for the simplicity that $Q_\varphi=0$, i.e.,  
 $\F_\varphi=\F_e$. Our proofs are essentially based on the
frame's version of the Naimark dilation theorem  proved by Han and Larson
(see \cite{hanlar} or Theorem \ref{Naimark}). The result of Theorem \ref{NEW1} 
is close to \cite[Proposition 2.4]{DH}, where an inverse problem: 
\emph{given a frame $\F_\psi$ with the frame operator $S$, under 
what conditions on $S$ does $\F_\psi$ admit a Parseval dual?} was investigated by
other methods.
Dual frames of a general frame $\F_\varphi$ are described in Section \ref{sec2.5} 
by the reduction to the results of Section \ref{sec2.2}. 

The parameters $\{Q, \F_e'\}$ in formulas for dual frames (Corollary \ref{DDD2b}, Theorem \ref{thm5}) 
can be easily calculated in concrete situations. 
We illustrate this fact by constructing dual frames for a PF with finite potential 
and excess one,  and for the Casazza-Christensen frame in Sections \ref{Sec2.3}, \ref{sec2.4b}. 

In Section \ref{sec3.2}, for the case of near-Riesz bases, we improve the general formula \eqref{NNN960} 
of dual frames by specifying the complementary PF
$\F_m$ in Theorem \ref{Naimark}.
For this reason, in Section \ref{sec3.1}, we establish a version of Naimark dilation theorem for near-Riesz bases.
The corresponding result -- Theorem \ref{Naimark2} fits well the expectations: 
the complementary PF $\F_m$ should be determined in major by a Riesz basis part of the original PF $\F_e$.  

 Throughout the paper,  $\K$ means  a complex Hilbert space with scalar product $ (\cdot, \cdot)$ 
linear in the first argument.  All operators in $\K$ are supposed to be linear, the identity operator is denoted by $I$.
The index set $\mathbb{J}$ is countable (or finite) and $|\mathbb{J}|$ means its cardinality.
 
\section{Dual frames}
\subsection{Preliminaries}\label{sec2.1}
Here all necessary information about frame theory are presented in a form convenient for our exposition.
More details  can be found in \cite{chri, hanlar, heil}.

A set of vectors $\F_\varphi=\{\varphi_j, \,j\in \mathbb{J}\}$ is called a \emph{frame} in $\K$ 
if there are constants $A$ and $B$, $0<A\leq B<\infty$, such that, for all $f\in\K$, 
\begin{equation}\label{K3}
A\|f\|^2\leq\sum_{j\in\mathbb{J}}|(f, \varphi_j)|^2\leq{B}\|f\|^2.
\end{equation}

The optimal constants in \eqref{K3}  (maximal for $A$ and minimal for $B$) are
called the frame bounds.  The \emph{frame operator} $Sf=\sum_{j\in\mathbb{J}}(f, \varphi_j)\varphi_j$ associated with $\F_\varphi$   is uniformly positive  in $\K$. 

A frame $\F_\varphi$ is called \emph{$A$-tight} if $A=B$.  
The formula \eqref{K3} for $1$-tight frames is similar to the Parseval equality for orthonormal bases: 
$\sum_{j\in\mathbb{J}}{|(f, \varphi_j)|^2}=||f||^2$ and the corresponding  $1$-tight frame $\F_\varphi$ is called  \emph{a  Parseval frame} (PF in the following). 
 PF's can be considered  as a generalization of orthonormal bases $\F_{e}=\{e_j, j\in\mathbb{J}\}$.
For this reason we will try to use the same notation $\F_{e}$ for  PF's.
 
  It is important that each frame can be reduced to a PF: 
 \emph{if   $\F_\varphi$ is a frame with frame operator $S$,  then 
 $\F_e=S^{-1/2}\F_\varphi=\{e_j=S^{- 1/2}\varphi_j\}$ is a PF.}
Denoting $S=e^Q$ 
we reformulate and extend this result in a form convenient for our presentation:
\begin{prop}\label{WW1}
The set $\F_\varphi=\{\varphi_j, j\in\mathbb{J} \}$ is a frame if and only if  there exists a bounded self-adjoint operator $Q$ in $\K$ 
and a PF $\F_e=\{e_j,  j\in\mathbb{J}\}$  such that
\begin{equation}\label{MM1}
 \F_{\varphi}=e^{Q/2}\F_e.
 \end{equation}

The PF $\F_e$ in \eqref{MM1} is an orthonormal basis in $\K$ if and only if  $\F_{\varphi}$ is a Riesz basis. 
 The frame operator of  $\F_{\varphi}$  coincides with $S=e^Q$ and
the canonical dual frame has the form  $\F_\psi=e^{-Q/2}\F_{e}$. 
\end{prop}

\begin{rem}\label{AGHWMS16}
In contrast to the familiar description of a frame $\F_\varphi=U\F_e$ (see, e.g., \cite[Theorem 5.5.5]{chri}), where an orthonormal basis $\F_e$ and a bounded and surjective operator
$U$ can be chosen variously, the self-adjoint operator $Q$ and the PF $\F_e$ in \eqref{MM1}
\underline{are determined uniquely}. Indeed, \eqref{MM1} determines a frame with the frame operator $S=e^Q$ because,
$$
Sf=\sum_{j\in\mathbb{J}}(f, \varphi_j)\varphi_j=\sum_{j\in\mathbb{J}}(f, e^{Q/2}e_j)e^{Q/2}e_j=e^{Q/2}\sum_{j\in\mathbb{J}}(e^{Q/2}f, e_j)e_j=e^{Q/2}e^{Q/2}f=e^{Q}f.
$$
Therefore, $e^{Q/2}$ in \eqref{MM1} is determined uniquely as the square root of $S$.  Then, the PF $\F_e$ is also determined uniquely as 
$\F_e=e^{-Q/2}\F_{\varphi}$.
\end{rem}

In view of Proposition \ref{WW1} and Remark \ref{AGHWMS16} the set of frames in $\K$ is in 
one-to-one correspondence with the set of pairs $\{Q, \F_e\}$, where $Q$ is a bounded self-adjoint
operator in $\K$, while $\F_e$ is a PF in $\K$.  For this reason, one can expect that
properties of a frame $\F_\phi$ can be expressed in terms of the corresponding pair
$\{Q, \F_e\}$. For example, the operator $Q$ determines the frame bounds and the potential
of $\F_\varphi$, while the excess of $\F_\varphi$ is characterized by the PF $\F_e$.
Namely, in view of \eqref{K3} and \eqref{MM1}, the frame bounds of $\F_\varphi$ are 
$A={1}/{\|e^{-Q/2}\|^2}$ and   $B=\|e^{Q/2}\|^2$.  
The potential of a frame 
$\F_\varphi=\{\varphi_j, \,j\in \mathbb{J}\}$ is defined by 
$$
{\bf FP}[\F_\varphi]=\sum_{j,i\in\mathbb{J}}|(\varphi_i, \varphi_j)|^2
$$
and, if $\dim\K<\infty$,
\begin{equation}\label{AGHWMS2}
{\bf FP}[\F_\varphi]=\sum_{n=1}^{\dim\K}\lambda_n^4,
\end{equation}
 where $\lambda_n$ are eigenvalues of $e^{Q/2}$  \cite{BF, Caz}. 

Following \cite{BCCL}, 
we recall that the \emph{excess} $e[\F_\varphi]$ of $\F_\varphi$  is the greatest integer $n$ such that  $n$ elements  
can be deleted from the frame $\F_\varphi$ and still leave a complete set, or $\infty$ if there is no upper bound to the number of elements that can be removed.

\begin{lemma}\label{WW1b}
The excess of a frame $\F_\varphi=e^{Q/2}\F_e$  coincides with the excess of $\F_e$.  
\end{lemma}
\begin{proof}
Each frame  $\F_\varphi$ determines an analysis operator $\theta_\varphi : \K \to \ell_2(\mathbb{J})$:
$$
\theta_\varphi{f}=\{(f, \varphi_j)\}_{j\in\mathbb{J}}, \qquad f\in\K.
$$
The image ${\mathcal R}(\theta_\varphi)$ of $\theta_\varphi$ coincides with the image  ${\mathcal R}(\theta_e)$ of $\theta_e$, 
since $(f, \varphi_j)=(f, e^{Q/2}e_j)=(e^{Q/2}f, e_j)$.  Due to \cite[Lemma 4.1]{BCCL}, 
$e[\F_\varphi]=\dim[\ell_2(\mathbb{J})\ominus{\mathcal R}(\theta_\varphi)]=\dim[\ell_2(\mathbb{J})\ominus{\mathcal R}(\theta_e)]=e[\F_e]$.
\end{proof}

Our approach to the description of dual frames is based on the Naimark dilation theorem  established in \cite[Chapter 1]{hanlar}.
\begin{thm}\label{Naimark}
Let $\F_e=\{e_j, \,j\in \mathbb{J}\}$ be a PF in a Hilbert space $\K$. Then there exists a Hilbert space $\M$ and a complementary PF $\F_{m}=\{{m}_j, \,j\in \mathbb{J}\}$ in $\M$ such that
\begin{equation}\label{K3b}
\F_{h}=\F_{e\oplus{m}}=\{h_j=e_j\oplus{m}_j, \  j\in \mathbb{J}\}
\end{equation} 
is an orthonormal basis for $\Hil=\K\oplus\M$. 
The extension of a PF $\F_e$ to an orthonormal basis  $\F_{h}=\F_{e\oplus{m}}$ described above is unique up to unitary equivalence.
\end{thm}

\begin{rem}\label{rem2}
It follows from \cite[Lemma 4.1]{BCCL}, the proof of  \cite[Proposition 1.1]{hanlar}, and the relation \eqref{AGHWMS2} that
the dimension of spaces $\K$ and $\M$ in Theorem \ref{Naimark} coincide,  respectively, with the potential ${\bf FP}[\F_e]$ and with the excess 
$e[\F_{e}]$ of $\F_{e}$. Moreover, $\dim\Hil=|\mathbb{J}|={\bf FP}[\F_e]+e[\F_{e}]$.
\end{rem}

\subsection{Dual frames for a Parseval frame}\label{sec2.2}
Let $\F_e=\{e_j, \,j\in \mathbb{J}\}$ be a Parseval frame in $\K$. 
According to Proposition \ref{WW1}, each dual frame $\F_\psi=\{\psi_j, \,j\in \mathbb{J}\}$ of  $\F_e$
can be presented as  $\F_\psi=e^{Q/2}\F_{e'}$, where $Q$ and $\F_{e'}$ are uniquely determined by
 $\F_\psi$. \emph{Could we specify conditions imposed on $Q$?} A similar question
inspired by an inverse problem: \emph{given a frame $\F_\psi$ with the frame operator $S$, under 
what conditions on $S$ does $\F_\psi$ admit a Parseval dual?} was investigated in \cite[Proposition 2.4]{DH}.

\begin{thm}\label{NEW1}
Given a  PF $\F_e$ in $\K$ with the excess $e[\F_e]$. 
 If $\F_\psi=e^{Q/2}\F_{e'}$  is a dual frame of $\F_e$, then the bounded operator $Q$ is 
 nonnegative\footnote{a bounded operator $Q$ is called  nonnegative if $(Qf, f) \geq 0$ for $f\in\K$} and
the relation
\begin{equation}\label{NEW2}
\dim\mathcal{R}(I-e^{-Q})\leq{e[\F_e]}
\end{equation}
holds. Conversely, if a bounded  operator $Q$ is nonnegative in $\K$  and \eqref{NEW2} holds,  
then there exists  a PF $\F_{{e'}}=\{{e}_j', \, j\in \mathbb{J}\}$ such that  
$\F_\psi=e^{Q/2}\F_{e'}$  is a dual frame of $\F_e$.
\end{thm}
\begin{proof}
Let $\F_\psi=e^{Q/2}\F_{e'}$  be a dual frame of $\F_e$. By Theorem \ref{Naimark}, the PF $\F_{e'}=\{{e}_j', j\in \mathbb{J}\}$  can be extended to
an orthonormal basis  $\F_{h'}=\{h_j'=e_j'\oplus{m_j'}, j\in \mathbb{J}\}$ in a Hilbert space $\Hil'$ containing $\K$ as a subspace. By the construction, $\dim\Hil=|\mathbb{J}|=\dim\Hil'$. 
Hence, the spaces  $\Hil$ and  $\Hil^o$ can be identified and the linear operator $W$ acting on vectors of the orthonormal basis $\F_{h}$ (see \eqref{K3b})  as
\begin{equation}\label{AGHWMS29} 
 Wh_j=W(e_j\oplus{m}_j)=e_j'\oplus{m}_j'=h_j', \qquad j\in\mathbb{J} 
 \end{equation}
 is unitary in  $\Hil$.  
 
 Denote $T=PW|_{\K}$, where  $P$ is the orthogonal 
projection operator in $\Hil$ onto the subspace $\K$.  Since $\F_\psi$ is dual for $\F_e$,
 $$
 f=\sum(f, e_j)\psi_j=e^{Q/2}\sum(f, e_j)e_j'=e^{Q/2}P\sum(f, e_j)h_j'=e^{Q/2}PW\sum(f, h_j)h_j=e^{Q/2}PWf
 $$ 
for all $f\in\K$. Therefore,  $e^{Q/2}T=I$ and $T=e^{-Q/2}$. By the construction, $T$ is a contraction in $\K$. Hence, $Q$ has to be a nonnegative bounded operator.

With  respect to the decomposition $\Hil=\K\oplus\M$ (Theorem \ref{Naimark}) the operator  $W$ has the form
\begin{equation}\label{DDD1}
W=\left[\begin{array}{cc}
e^{-Q/2} &  W_{12} \\
W_{21}  &  W_{22}
\end{array}\right],
\end{equation}  
where  $W_{21} : \K\to\M$ and $W_{12} : \M\to\K$. The operator coefficients $W_{ij}$ satisfy additional restrictions  
since $W$ is unitary. In particular, considering $W$ on elements $f\oplus{0}$, we get $((e^{-Q}+W_{21}^*W_{21})f, f)=(f,f)$.
Therefore, 
$$
I-e^{-Q}=W_{21}^*W_{21} \quad   \mbox{and} \quad  \ker(I-e^{-Q})=\ker{W_{21}}. 
$$
 Using the decomposition
$\K=\overline{\mathcal{R}(I-e^{-Q})}\oplus\ker(I-e^{-Q})$,  we arrive at the conclusion that  $W_{21} :\overline{\mathcal{R}(I-e^{-Q})} \to\M$
is an injective  mapping. Hence, $\dim\overline{\mathcal{R}(I-e^{-Q})}\leq{\dim\M}$ that justifies \eqref{NEW2}  due to Remark \ref{rem2}.  

Conversely, given a bounded nonnegative operator $Q$ satisfying \eqref{NEW2}. The nonnegativity of $Q$ implies that $T=e^{-Q/2}$ is 
a self-adjoint contraction in $\K$.  Assume that there exists a unitary operator $W$ in $\Hil$ such that $T=PW|_{\K}$.
Then, the set $\F_{h'}=\{h_j'=Wh_j\}$ is an orthonormal basis of $\Hil$ and  $\F_{e'}=\{e_j'=P{h}_j'\}=P\F_{h'}$
is a PF in $\K$. Denote $\F_\psi=e^{Q/2}\F_{e'}=\{\psi_j=e^{Q/2}e_j'\}$. Since
$$
\sum|(f, \psi_j)|^2=\sum|(e^{Q/2}f, e_j')|^2=\|e^{Q/2}f\|^2\leq\|e^{Q/2}\|^2\|f\|^2, \quad f\in\K,
$$
the set $\F_\psi$ is a Bessel sequence. Furthermore, 
$$
\sum(f, e_j)\psi_j=e^{Q/2}P\sum(f, e_j)h_j'=e^{Q/2}PW\sum(f, h_j)h_j=e^{Q/2}PWf=e^{Q/2}e^{-Q/2}f=f
$$
and, hence, $\F_\psi$ is a dual frame of $\F_e$.  

In order to compete the proof one should 
verify that the condition \eqref{NEW2} guarantees  that $e^{-Q/2}$ can be expanded to an unitary operator $W$ in the Hilbert space
$\Hil=\K\oplus\M$ such that $e^{-Q/2}=PW|_{\K}$.
It is easy to see that the operator
\begin{equation}\label{eq:ww}
W=\left[\begin{array}{cc}
e^{-Q/2} &  (I-e^{-Q})^{1/2} \\
(I-e^{-Q})^{1/2}  &  -e^{-Q/2} 
\end{array}\right],
\end{equation}
is unitary in the Hilbert space $\K\oplus\overline{\mathcal{R}(I-e^{-Q})}$ and $e^{-Q/2}=PW|_{\K}$.  If  \eqref{NEW2} holds, then
$\overline{\mathcal{R}(I-e^{-Q})}$ can be considered as a subspace of $\M$.  Defining $W$ as the identity operator
on $\M\ominus\overline{\mathcal{R}(I-e^{-Q})}$ we obtain the required unitary operator in $\Hil$.
\end{proof}

\begin{cor}\label{DDD2b}
A dual frame $\F_\psi=e^{Q/2}\F_{e'}$  of $\F_e$ considered in Theorem \ref{NEW1} consists of vectors
\begin{equation}\label{AGHWMS101}
\psi_j=e_j+e^{Q/2}W_{12}m_j,  \qquad j\in\mathbb{J}.
\end{equation}
where $W_{12}:\M\to\K$ is a part of an unitary operator \eqref{DDD1} and $\{m_j\}$ 
is the complementary PF of $\F_e$ in $\M$  (see \eqref{K3b}). 
The canonical dual frame of $\F_e$ corresponds to the operator $Q=0$ in \eqref{AGHWMS101}.
\end{cor}
\begin{proof} In view of the proof above and \eqref{K3b}, the vectors $e_j'$ of the PF $\F_{e'}$
have the form $e_j'=PWh_j=e^{-Q/2}e_j+W_{12}m_j$ that justifies \eqref{AGHWMS101}. The operator $Q=0$ 
satisfies \eqref{NEW2} and $T=e^{-Q/2}$ is the identity operator in $\K$. This means that the unitary dilation
$W$ of $T=I$ in $\Hil$ has the form \eqref{DDD1}, where $W_{12}=W_{21}=0$. Applying \eqref{AGHWMS101} we obtain
that the dual frame $\F_\psi$ coincides with $\F_e$. Therefore the zero operator $Q$ determines 
the canonical dual frame of $\F_e$.
\end{proof}

Theorem \ref{NEW1} explains which conditions satisfy the operator $Q$ in the formula $\F_\psi=e^{Q/2}\F_{e'}$ for dual frames.  Analogously, the second counterpart of this formula --  a PF $\F_{{e'}}$ cannot be selected arbitrarily and its choice should satisfy certain restrictions.
 \begin{cor}\label{AGHWMS50}
 Let  $\F_{{e'}}$ be a PF in $\K$. The following are equivalent: 
\begin{enumerate}
\item[(a)] there exists a nonnegative operator $Q$ such that the formula
$\F_\psi=e^{Q/2}\F_{e'}$ determines a dual frame for $\F_e$;
\item[(b)] the excess of 
$\F_{{e'}}$  coincides with the excess of the original PF $\F_e$ and the operator $PW|_{\K}$, where $W$ is defined by
\eqref{AGHWMS29},  is  strongly positive in $\K$.
\end{enumerate}
 \end{cor}
 \begin{proof} $(a)\to(b)$. In view of \cite[Theorem 2.2]{BB}, the excess of a dual frame $\F_\psi$ coincides with the excess of $\F_e$.
 Due to Lemma \ref{WW1b},  $e[\F_\psi]=e[\F_{e'}]$ that gives $e[\F_{e'}]=e[\F_e]$.  By the proof of Theorem \ref{NEW1},
 the operator $PW|_{\K}$ coincides with $e^{-Q/2}$ and, hence, it is strongly positive.
 
 $(b)\to(a)$. Since $e[\F_{e'}]=e[\F_e]$, the auxiliary subspaces $\M$  for PF's $\F_{e'}$ and $\F_{e}$ in Theorem \ref{Naimark}   have the same
 dimension and, therefore these subspaces can be identified. This means that the unitary operator $W$ in \eqref{AGHWMS29} is well posed.
 Denoting $e^{-Q/2}=PW|_{\K}$ and repeating the proof of Theorem \ref{NEW1} we complete the proof.
\end{proof}
 
 \subsection{Dual frames for a PF with finite potential and excess one}\label{Sec2.3}
 Let  $\F_e$ be a PF in a Hilbert space $\K$ with ${\bf FP}[\F_e]=K\in\mathbb{N}$ and $e[\F_e]=1$.
 This means that $\dim{\K}=K$ and $\dim\Hil=K+1$, where $\Hil=\K\oplus\M$ is a Hilbert space in Theorem \ref{Naimark} and
 $|\mathbb{J}|=K+1$. Without loss of generality, we assume that $\mathbb{J}=\{1, 2\ldots, K+1\}$,  $\K=\mathbb{C}^K$, 
 $\M=\mathbb{C}$ and
 $\Hil=\mathbb{C}^{K}\oplus\mathbb{C}=\mathbb{C}^{K+1}$.  

Our aim is the description of all dual frames of $\F_e$. Due to Theorem \ref{NEW1}, one should consider a
nonnegative operator $Q$. Denote $T=e^{-Q/2}$. By the construction, $T$ is a positive contraction operator in $\mathbb{C}^K$
and the inequality  \eqref{NEW2} takes the form
$$
 \dim\mathcal{R}(I-e^{-Q})=\dim\mathcal{R}(I-T)(I+T)=\dim\mathcal{R}(I-T)\leq{e[\F_e]}=1.
 $$
 
 If $\dim\mathcal{R}(I-T)=0$, then $I=T$ and $Q=0$ that corresponds to the case of canonical dual frame (Corollary \ref{DDD2b}). Assume that $\dim\mathcal{R}(I-T)=1$. In this case, the positive 
contraction $T$ has a simple eigenvalue $0<\varepsilon<{1}$ and the eigenvalue $1$ of the multiplicity $K-1$.
A rudimentary linear algebra exercise leads to the conclusion that such operators $T$ in $\mathbb{C}^K$
are described by the matrices depending on  the choice of  a normalized eigenvector
$u=(u_1, u_2, \ldots u_K)^t$ of $T$ corresponding to the eigenvalue $\varepsilon$: 
\begin{equation}\label{AGHWMS1}
T=e^{-Q/2}=\left[\begin{matrix} 1+(\varepsilon-1)|u_1|^2  &   (\varepsilon-1){u_1}\overline{u}_2  & \ldots  & (\varepsilon-1){u_1}\overline{u}_K \\
(\varepsilon-1){u_2}\overline{u}_1  &  1+(\varepsilon-1)|u_2|^2 &  \ldots  & (\varepsilon-1){u_2}\overline{u}_K \\
\vdots &  \vdots & \vdots  & \vdots  \\
(\varepsilon-1){u_K}\overline{u}_1  &  (\varepsilon-1){u_K}\overline{u}_2 &  \ldots  &  1+(\varepsilon-1)|u_K|^2
\end{matrix}\right].
\end{equation}

The choice of $\varepsilon=1$ in \eqref{AGHWMS1} corresponds to the case $T=I$ (i.e., the case of canonical dual frame). For this reason, in what follows, we assume that $0<\varepsilon\leq{1}$. 

According to Corollary \ref{DDD2b}, each dual frame $\F_\psi$ 
is described by \eqref{AGHWMS101}, where
$W_{12}$ is a component of an unitary dilation \eqref{DDD1} of
$T$.  It is easy to check that required unitary operators in $\mathbb{C}^{K+1}=\mathbb{C}^K\oplus\mathbb{C}$ are 
determined by the matrices:
\begin{equation}\label{HH3b}
W=\left[\begin{matrix}
1+(\varepsilon-1)|u_1|^2  &   (\varepsilon-1){u_1}\overline{u}_2  & \ldots  & (\varepsilon-1){u_1}\overline{u}_K   &  \sqrt{1- \varepsilon^2}u_1e^{i\theta} \\
(\varepsilon-1){u_2}\overline{u}_1  &  1+(\varepsilon-1)|u_2|^2 &  \ldots  & (\varepsilon-1){u_2}\overline{u}_K  & \sqrt{1- \varepsilon^2}u_2e^{i\theta} \\
\vdots &  \vdots & \vdots  & \vdots  & \vdots \\
(\varepsilon-1){u_k}\overline{u}_1  &  (\varepsilon-1){u}_k\overline{u}_2 &  \ldots  &  1+(\varepsilon-1)|u_K|^2 & \sqrt{1- \varepsilon^2}u_Ke^{i\theta}\\
\sqrt{1- \varepsilon^2}\overline{u}_1e^{i\tilde{\theta}} &  \sqrt{1- \varepsilon^2}\overline{u}_2e^{i\tilde{\theta}} & \ldots  & \sqrt{1- \varepsilon^2}\overline{u}_Ke^{i\tilde{\theta}}  & -{\varepsilon}e^{i(\theta+\tilde{\theta})}
\end{matrix}\right],
\end{equation}
where $\theta, \tilde{\theta}\in[0, 2\pi)$.  

By virtue of \eqref{DDD1} and \eqref{HH3b}, the operator $W_{12} : \M=0\oplus\mathbb{C}\to\K=\mathbb{C}^{K}\oplus{0}$ acts as the multiplication by
$\sqrt{1- \varepsilon^2}{e^{i\theta}}\left[\begin{array}{c}
u_1 \\
\vdots \\
u_k
\end{array}
\right]$. According to  \eqref{AGHWMS101},  a dual frame $\F_\psi=\{\psi_1, \ldots \psi_{K+1}\}$
consists of the vectors: 
$$
\psi_j=e_j+T^{-1}W_{12}m_j=e_j+\sqrt{1- \varepsilon^2}{e^{i\theta}}T^{-1}\left[\begin{array}{c}
u_1 \\
\vdots \\
u_k
\end{array}
\right]m_j,
$$
where $\{m_j\}$ is the complementary PF in $\mathbb{C}$ for $\F_e$  (see \eqref{K3b}). 
The obtained expression can be further simplified by taking into account the following observations:
$u=(u_1, \ldots u_K)^t$ is an eigenvector of $T^{-1}$ corresponding to the eigenvalue $1/\varepsilon$; 
the parameter ${e^{i\theta}}$ can be omitted if $u$ is assumed to be an arbitrary normalized vector in $\mathbb{C}^K$. Finally,  we obtain the following description of all dual frames  $\F_\psi=\{\psi_1, \ldots \psi_{K+1}\}$
of $\F_e$:
\begin{equation}\label{HH6}  
\psi_j=e_j+\frac{\sqrt{1-\varepsilon^2}}{\varepsilon}\left[\begin{array}{c}
u_1 \\
\vdots \\
u_K
\end{array}
\right]m_j, \qquad j=1, 2\ldots, K+1,
\end{equation}
where $0<\varepsilon\leq{1}$ and $(u_1, \ldots u_K)^t$ is an arbitrary normalized vector in $\mathbb{C}^K$.

The parameter  $\varepsilon$ in \eqref{HH6} determines the frame bounds of $\F_\psi$: $A={1}/{\|e^{-Q/2}\|^2}=1$ and   
$B=\|e^{Q/2}\|^2=\frac{1}{\varepsilon^2}$. Moreover, the potential of $\F_\psi$ is expressed
in terms of $\varepsilon$:
$$
{\bf FP}[\F_\psi]=K-1+(1/\varepsilon)^4.
$$
The last equality follows from \eqref{AGHWMS2} and the fact
that  the operator $T^{-1}=e^{Q/2}$ has eigenvalue $1$ of multiplicity $K-1$ and an eigenvalue $1/\varepsilon$.
The canonical dual frame corresponds to the case $\varepsilon=1$. For a given $0<\varepsilon<1$ the formula
\eqref{HH6} presents dual frames with the same potential and optimal bounds that are determined by 
the variation of normalized vector $u\in\mathbb{C}^K$.  

\begin{rem}
 The above approach can  be used for the construction of dual frames of a PF  with an arbitrary \emph{finite} potential and
 excess. For example, if   ${\bf FP}[\F_e]=n{\geq}m=e[\F_e]$, then the $n\times{n}$-matrix in \eqref{AGHWMS1} will be determined by 
 $m$ orthonormal eigenvectors $\{u(j)\in\mathbb{C}^n\}_{j=1}^m$ corresponding to eigenvalues $0<\varepsilon_1\leq\varepsilon_2\leq\ldots\leq\varepsilon_m\leq{1}$.
 Similarly, $(n+m)\times{(n+m)}$-unitary matrix \eqref{HH3b}  will contain more parameters.
\end{rem}
 
For a concrete PF $\F_e$ one can specify the complementary PF 
 $\{m_j\}$ in \eqref{HH6}.  
\begin{example} 
 Dual frames for Mercedes frame.
 \end{example}
 In  the Hilbert space $\Hil=\mathbb{C}^3$, we consider the orthonormal basis
$\F_{h}=\{h_1, h_2, h_3\}$, where 
$$
h_1=\sqrt{\frac{2}{3}}\left[
\begin{array}{c}
1 \\
0 \\
\frac{1}{\sqrt{2}}\\
\end{array}
\right], \qquad h_2=\sqrt{\frac{2}{3}}\left[
\begin{array}{c}
-\frac{1}{2} \\
\frac{\sqrt{3}}{2} \\
\frac{1}{\sqrt{2}} 
\end{array}
\right], \qquad h_3=\sqrt{\frac{2}{3}}\left[
\begin{array}{c}
-\frac{1}{2} \\
-\frac{\sqrt{3}}{2} \\
\frac{1}{\sqrt{2}}
\end{array}
\right].
$$

The orthogonal projection of $\F_h$ onto the subspace $\K=\mathbb{C}^2\oplus{0}$
determines  a PF $\F_e=\{e_j\}$, 
$$
e_1=\sqrt{\frac{2}{3}}\left[
\begin{array}{c}
1 \\
0 
\end{array}
\right], \qquad e_2=\sqrt{\frac{2}{3}}\left[
\begin{array}{c}
-\frac{1}{2} \\
\frac{\sqrt{3}}{2} 
\end{array}
\right], \qquad e_3=\sqrt{\frac{2}{3}}\left[
\begin{array}{c}
-\frac{1}{2} \\
 -\frac{\sqrt{3}}{2} 
\end{array}
\right],
$$
which is called the Mercedes frame \cite[p. 204]{heil}.  

By the construction, the excess of  $\F_e$ is $1$,  the potential ${\bf FP}[\F_e]=2$, and the
complementary PF 
$\F_m=\{m_1, m_2, m_3\}$ in \eqref{K3b} consists of the vectors (numbers) 
$m_j=(I-P)h_j=\frac{1}{\sqrt{3}}$ in $\M=\mathbb{C}$.

 Denoting normalized vectors $u\in\mathbb{C}^2$ as
$$
u=\left[\begin{array}{c}
e^{i\beta}\cos\alpha \\
e^{i\gamma}\sin\alpha
\end{array}
\right],  \qquad  \alpha, \beta, \gamma\in[0, 2\pi)
$$ 
and using \eqref{HH6} we obtain the description of all dual frames $\F_\psi=\{\psi_1, \psi_2, \psi_3  \}$ of the Mercedes frame $\F_e$ 
$$ 
\psi_j=e_j+\frac{\sqrt{1-\varepsilon^2}}{\sqrt{3}\varepsilon}\left[\begin{array}{c}
e^{i\beta}\cos\alpha \\
e^{i\gamma}\sin\alpha
\end{array}
\right] 
$$
that involves four parameters: $\alpha, \beta, \gamma\in[0, 2\pi)$ and $0<\varepsilon\leq{1}$. The dual frame $\F_\psi$ has the frame bounds $A=1$, $B=\frac{1}{\varepsilon^2}$.
The potential of $\F_\psi$ is $1+\frac{1}{\varepsilon^4}$. 

\begin{example} 
 SIC-POVM  represented by the Bloch-Sphere.
 \end{example}
SIC-POVM's (symmetric informationally-complete positive operator value measures) are an important class of generalized measure that are used in quantum measurement
 theory \cite{RBSC}. 
For the Hilbert space $\mathbb C^n$, it is defined by a set of $n^2$  normalized  vectors $\{e_j\}_{j=1}^{n^2} \subset \mathbb C^n$,  with  a fixed inner product, 
$|(e_i,e_j)|^2_{i\not=j}=\frac{1}{n+1}$ . We consider $n=2$ and the qubit $\left[\begin{array}{c}
1 \\
0 
\end{array}
\right] \in \mathbb C^2$, due to the properties of SIC-POVM, one can associate 3 vectors from $\mathbb C^2$ to  obtain  the following SIC-POVM 
$$
e_1=\left[
\begin{array}{c}
1 \\
0 
\end{array}
\right], \qquad e_2=\left[
\begin{array}{c}
\frac{1}{\sqrt{3}} \\
\sqrt{\frac{2}{3}} 
\end{array}
\right], \qquad e_3=\left[
\begin{array}{c}
\frac{1}{\sqrt{3}}  \\
 \sqrt{\frac{2}{3}}e^{2i \pi/3}
\end{array}
\right],
\qquad e_4=\left[
\begin{array}{c}
\frac{1}{\sqrt{3}}  \\
 \sqrt{\frac{2}{3}}e^{-2i \pi/3}
\end{array}
\right].
$$

 In order to  visualize SIC-POVM in $\mathbb C^2$  we use the Bloch-Sphere.  For a quantum state $e_i=\left[
\begin{array}{c}
{e_{i,1}} \\
{e_{i,2}} 
\end{array}
\right]$, a density operator $\rho$ is defined as follows
\begin{equation}\label{eq:sicpovm1}
 \rho= \left[\begin{matrix}|e_{i,1}|^2 & e_{i,1}\overline{e}_{i,2} \\
 e_{i,2}\overline{e}_{i,1} & |e_{i,2}|^2 \end{matrix}\right].
\end{equation}

On the other hand, $\rho$ can be expressed via the Pauli matrices and the identity matrix  
\begin{equation}\label{eq:sicpovm2}
\rho=\frac{1}{2}\left( \left[\begin{matrix}1 & 0\\0 & 1\end{matrix}\right]+x  \left[\begin{matrix}0 & 1\\1 & 0\end{matrix}\right]+y  \left[\begin{matrix}0 & - i\\i & 0\end{matrix}\right]+z \left[\begin{matrix}1 & 0\\0 & -1\end{matrix}\right]\right).
\end{equation}
   Combining (\ref{eq:sicpovm1}) and (\ref{eq:sicpovm2}) and iterating for $j=1,\dots,4$, 
   each solution $x,y,z$ gives rise to a vector $f_j$ from the Bloch-Sphere  corresponding to $e_j$ \cite[Chapter 1.1.1]{MR}:
 
$$ f_1=\left[\begin{matrix}0\\0\\1\end{matrix}\right], \qquad f_2=  \left[\begin{matrix}\frac{2 \sqrt{2}}{3}\\0\\- \frac{1}{3}\end{matrix}\right], \qquad f_3=   \left[\begin{matrix}- \frac{\sqrt{2}}{3}\\\frac{\sqrt{6}}{3}\\- \frac{1}{3}\end{matrix}\right], \qquad f_4=  \left[\begin{matrix}- \frac{\sqrt{2}}{3}\\- \frac{\sqrt{6}}{3}\\- \frac{1}{3}\end{matrix}\right].
 $$
 
 Simple calculation shows that $\F_f=\{f_j\}_{j=1}^4$ is a $\frac{4}{3}$-tight frame with the excess $e[\F_f]=1$. 
The complementary PF $\{m_j\}_{j=1}^4$   for the PF $\F_e=\frac{\sqrt{3}}{2}\F_f$ consists of vectors:
$m_1=\ldots{m}_4=\frac{1}{2}$. Dual frames for $\F_e$ are described by  (\ref{HH6}). 
Multiplying this expression by $\frac{\sqrt{3}}{2}$ we obtain dual frames
 $\F_\psi=\{\psi_j\}_{j=1}^4$ for $\F_f$: 
$$
\psi_j=\frac{\sqrt{3}}{2}\left(e_j+\frac{\sqrt{1-\varepsilon^2}}{2\varepsilon}u\right),
$$
where $u=(u_1, u_2, u_3)^t$ is an arbitrary normalized vector in $\mathbb{C}^3$.

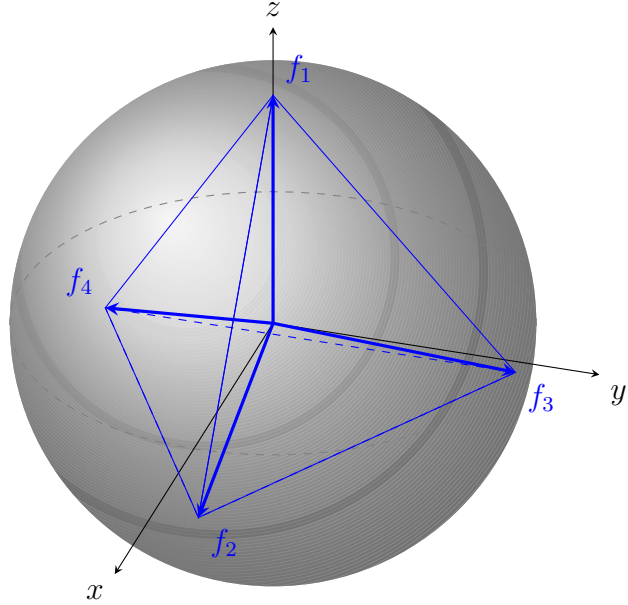
\begin{figure}
\centering
\begin{tikzpicture}[tdplot_main_coords, scale = 3.5, rotate around z=7.5]
\coordinate (Q) at ({0},{0},{1}); 
\coordinate (R) at ({2*sqrt(2)/3},{0},{-1/3}); 
\coordinate (S) at ({-sqrt(2)/3},{sqrt(6)/3},{-1/3}); 
\coordinate (T) at ({-sqrt(2)/3},{-sqrt(6)/3},{-1/3});

\shade[ball color = lightgray, opacity = 0.4, thick] (0,0,0) circle (1cm);
 
\tdplotsetrotatedcoords{0}{0}{0};
\draw[dashed, tdplot_rotated_coords, gray] (0,0,0) circle (1);
 
\draw[-stealth] (0,0,0) -- (2.0,0,0) node[below left] {$x$};
\draw[-stealth] (0,0,0) -- (0,1.30,0) node[below right] {$y$};
\draw[-stealth] (0,0,0) -- (0,0,1.30) node[above] {$z$};

\draw[very thick, -stealth,color=blue] (0,0,0) -- (Q) node[above right] {$f_1$};
\draw[very thick, -stealth,color=blue] (0,0,0) -- (R) node[below right] {$f_2$};
\draw[very thick, -stealth,color=blue] (0,0,0) -- (S) node[below right] {$f_3$};
\draw[very thick, -stealth,color=blue] (0,0,0) -- (T) node[above left] {$f_4$};
\draw[color=blue] (Q) -- (R) -- (S) -- cycle;
\draw[color=blue ] (R) -- (Q) -- (T) -- cycle;
 \draw[color=blue, dashed ] (S) -- (T) -- (R) -- cycle;
\end{tikzpicture}
\caption{Four vectors of SIC-POVM form a tetrahedron inside the Bloch-Sphere.}
 \end{figure}
  
  \subsection{Dual frames for the Casazza-Christensen frame}\label{sec2.4b}
In general, a PF with infinite excess need not contain a Riesz basis as a subset.  It is easy to construct a frame which does not contain a Riesz basis
if one allows a subsequence of the frame elements to converge to $0$ in norm. For example,  if
 $\{b_j\}_{j=1}^\infty$ is an orthonormal basis, then the PF  
 $$
 \F_{b}=\{ b_1, \frac{1}{\sqrt{2}}b_2, \frac{1}{\sqrt{2}}b_2,  \frac{1}{\sqrt{3}}b_3, \frac{1}{\sqrt{3}}b_3, \frac{1}{\sqrt{3}}b_3, \ldots \}
 $$
 does not contain a Riesz basis at all. The only candidate would be $\{\frac{1}{\sqrt{j}}b_j\}_{j=1}^\infty$,  which is a Schauder basis but not a Riesz basis.
 It turns out that there exist PF's which consist of vectors that are norm bounded below, but which do not contain a Shauder basis and, therefore they do not contain
a Riesz basis. The relevant curious example was constructed by Casazza and Christensen \cite{CC98}. 
Aiming to describe dual frames for this PF we begin with the description of duals for  finite components of the Casazza-Christensen frame.

 Let $\{b_j\}_{j=1}^K$, $K\in\mathbb{N}$ be an orthonormal basis of a Hilbert space $\K$.
 The set $\F_{e}=\{e_j, j\in\mathbb{J}\}$, where $\mathbb{J}=\{1, 2,\ldots, K+1\}$ and
\begin{equation}\label{K61}
e_j=b_j-\frac{1}{K}\sum_{i=1}^{K}b_i,  \quad 1\leq{j}\leq{K}, \qquad e_{K+1}=\frac{1}{\sqrt{K}}\sum_{i=1}^{K}{b}_i,
\end{equation}
is a PF in $\K$ \cite[Lemma 2.5]{CC98}.  It is easy to see that ${\bf FP}[\F_e]=K$  and  $e[\F_e]=1$. Hence, the dimension of the auxiliary space
$\M$ in  Theorem \ref{Naimark} is $1$ and, without loss of generality we may assume
that $\M=\mathbb{C}$ and $\Hil=\K\oplus\mathbb{C}$.  

It follows from \eqref{K61} that $\|e_j\|^2=1-\frac{1}{K}$, \  $\|e_{K+1}\|^2=1$, and
$$
(e_j, e_p)=-\frac{1}{K}, \quad j\not=p\in\{1,\ldots, K\}; \qquad (e_j, e_{K+1})=0, \quad j\leq{K}.
$$ 
 For this reason, the complementary PF 
$\F_m=\{m_1, \ldots, m_{K+1}\}$ in \eqref{K3b} consists of the vectors 
\begin{equation*}
m_j=\frac{1}{\sqrt{K}}, \quad j=1,\ldots,K,   \qquad {m}_{K+1}=0.
\end{equation*}

By virtue of \eqref{HH6}, dual frames $\F_\psi=\{\psi_1, \ldots \psi_{K+1}\}$ for $\F_e$ are formed by  the vectors: 
\begin{equation}\label{AGHWMS14}
\psi_j=e_j+\frac{\sqrt{1-\varepsilon^2}}{\sqrt{K}\varepsilon}u, \quad j=1,\ldots,K,  \qquad \psi_{K+1}=e_{K+1},
\end{equation}
where $u$ is an arbitrary normalized vector in $\K$.

Let $\K$ be a separable Hilbert space.  Index an orthonormal basis for $\K$ as $\{b_j^K\}_{j=1}^K$, where $K$ runs the set of natural numbers and 
put $\K_K=\mbox{span}\{b_1^K, b_2^K, \ldots b_K^K\}$.  The vectors $e_j^K\equiv{e_j}$ defined by \eqref{K61}
form a PF $\F_{e^K}=\{e_j^K, \ j=1, 2\ldots{K+1}\}$ of the Hilbert space $\K_K$.
Since $\K=\sum_{K=1}^\infty\oplus\K_K$, the collection of vectors 
\begin{equation}\label{GGG}
\F_{e}=\{e_1^K, e_2^K,\ldots, e_{K+1}^K\}_{K=1}^{\infty}=\bigcup_{K=1}^\infty\F_{e^K}
\end{equation}
is a PF for $\K$. It was shown  \cite{CC98} that this PF does not contain a Riesz/Schauder basis.  

Denote by $\F_\psi^K=\{\psi_1^K, \ldots \psi_{K+1}^K\}$ a dual frame for $\F_{e^K}$ in $\K_K$. 
By analogy with \eqref{AGHWMS14},
$$
\psi_j^K=e_j^K+\frac{\sqrt{1-\varepsilon^2_K}}{\sqrt{K}\varepsilon_K}u_K,  \quad \psi_{K+1}^K=e_{K+1}^K, \quad u_K\in\K_K, \ \|u_K\|=1.
$$

It follows from \eqref{GGG} that $\F_{\psi}=\{\psi_1^K, \psi_2^K,\ldots, \psi_{K+1}^K\}_{K=1}^{\infty}=\bigcup_{K=1}^\infty\F_{\psi^K}$
is a dual frame for  $\F_{e}$ in $\K$.
 
\subsection{Dual  frames for a general frame}\label{sec2.5}
Our aim is the description of dual frames for a general frame
$\F_\varphi$. The similar problem was considered in Section \ref{sec2.2}
for PF's. The following technical statement reduces the study of general case
to the results of Section \ref{sec2.2}.

Let  $\F_\varphi=\{\varphi_j, j\in\mathbb{J} \}$ be a frame in $\K$. 
By Proposition \ref{WW1}, $\F_{\varphi}=e^{Q_\varphi/2}\F_e$, where
$Q_\varphi$ is a bounded self-adjoint operator and $\F_e=\{e_j,  j\in\mathbb{J}\}$
is a PF in $\K$.
\begin{lemma}\label{NNN950}
The following are equivalent:
\begin{enumerate}
\item[(i)] a frame $\F_\psi=\{\psi_j, j\in\mathbb{J} \}$ is dual for  $\F_\varphi=e^{Q_\varphi/2}\F_{e}$;
\item[(ii)] a frame $e^{Q_\varphi/2}\F_\psi$ is 
dual for the Parseval frame $\F_{e}$.
\end{enumerate}
\end{lemma}
\begin{proof}
 The relation
$$
f=\sum(f, \varphi_j)\psi_j=\sum(f, e^{Q_\varphi/2}e_j)\psi_j=\sum(e^{Q_\varphi/2}f, e_j)\psi_j=\sum(g, e_j)\psi_j=e^{-Q_\varphi/2}g,
$$
where $g=e^{Q_\varphi/2}f$ implies $g=\sum(g, e_j)e^{Q_\varphi/2}\psi_j$ that justifies
$(i)\to(ii)$. The implication $(ii)\to(i)$ is proved similarly.
\end{proof}

Lemma \ref{NNN950} allows one to generalize  Corollary \ref{DDD2b}
for the case of an arbitrary frame $\F_\varphi$.

\begin{thm}\label{thm5}
 Given a frame $\F_\varphi=e^{Q_\varphi/2}\F_{e}$.  Each dual frame $\F_\psi$
 of $\F_\varphi$ consists of the vectors
\begin{equation}\label{NNN960}
\F_\psi=\{\psi_j=e^{-Q_\varphi/2}e_j+e^{-Q_\varphi/2}e^{Q/2}W_{12}m_j, \ j\in\mathbb{J}\},
\end{equation}
where $Q$ is a bounded nonnegative operator in $\K$ such that
\begin{equation}\label{NNN961}
\dim\mathcal{R}(I-e^{-Q})\leq{e[\F_\varphi]},
\end{equation}
where $W_{12}:\M\to\K$ is a part of an unitary operator \eqref{DDD1} and $\{m_j\}$ 
is the complementary PF of $\F_e$ in $\M$ (see \eqref{K3b}).
 \end{thm}
\begin{proof} 
 Let $\F_\psi$ be a dual frame  of $\F_\varphi=e^{Q_\varphi/2}\F_{e}$. 
 In view of  Lemma \ref{NNN950},  $e^{Q_\varphi/2}\F_\psi$ 
 is a dual frame for the PF $\F_{e}$. By Corollary \ref{DDD2b}, the frame 
 $e^{Q_\varphi/2}\F_\psi$ is formed by the vectors
$$
e^{Q_\varphi/2}\psi_j=e_j+e^{Q/2}W_{12}m_j,
$$
where a bounded nonnegative operator $Q$ satisfies  
  \eqref{NNN961} (here $e[\F_\varphi]=e[\F_e]$ in view of Lemma \ref{WW1b}). The last
 equality implies \eqref{NNN960}.
\end{proof}

\begin{rem}\label{KKK128}
In view of Proposition \ref{WW1},  $\{e^{-Q_\varphi/2}e_j\}$ is a canonical dual frame of  $\F_\varphi=e^{Q_\varphi/2}\F_{e}$. For this reason,  \eqref{NNN960} represents alternate dual 
frames as perturbations of the canonical one  that are parameterized by $Q$. The parameter
vanishes for the canonical dual frame (it follows from  Corollary \ref{DDD2b} that $Q=0$).
\end{rem}

By virtue of Lemma \ref{NNN950} and Theorem \ref{NEW1}, each dual frame $\F_\psi$ of $\F_\varphi=e^{Q_\varphi/2}\F_{e}$
can also be presented as
    \begin{equation}\label{DDD5}
     \F_\psi=e^{-Q_\varphi/2}e^{Q/2}\F_{e'}=R\F_{e'}, \qquad R=e^{-Q_\varphi/2}e^{Q/2},
    \end{equation}
 where an operator $Q$ and a PF $\F_{e'}$  are taken from Theorem \ref{NEW1}.

 The operator $R$ in \eqref{DDD5} is not necessarily self-adjoint in $\K$. For this reason the formula  ${\F}_\psi=R{{\F}}_{e'}$ differs from the standard presentation of frames ${\F}_\psi=e^{Q/2}{{\F}}_{e}$  (see Proposition \ref{WW1}), that is used consistently in the paper.
 Using the polar decomposition of $R=|R^*|U$ \cite[Chapter VI, (2.26)]{Kato}, where 
 $$
 |R^*|=\sqrt{RR^*}=\sqrt{e^{-Q_\varphi/2}e^{Q}e^{-Q_\varphi/2}}
 $$
 and $U$ is a unitary 
operator in $\K$, we obtain
\begin{equation}\label{KK23}
{\F}_\psi=R{{\F}}_{e'}=|R^*|U{{\F}}_{e'}=\sqrt{e^{-Q_\varphi/2}e^{Q}e^{-Q_\varphi/2}}\F_{Ue'}=e^{Q_\psi/2}\F_{Ue'},
\end{equation}
where $\F_{Ue'}=U\F_{e'}$ is a PF and $e^{Q_\psi/2}=\sqrt{e^{-Q_\varphi/2}e^{Q}e^{-Q_\varphi/2}}$ is a  self-adjoint operator in $\K$.

The formula \eqref{KK23} is simplified if $Q$ commutes with $Q_\varphi$. Then 
$R=e^{-Q_\varphi/2}e^{Q/2}=e^{(Q-Q_\varphi)/2}$ and it is positive. This yields
$U=I$ and the formula \eqref{KK23} is rewritten as
\begin{equation}\label{KKK1}
{\F}_\psi=e^{(Q-Q_\varphi)/2}{\F}_{e'}.
\end{equation}

\begin{cor}\label{cor1}
The following are equivalent:
\begin{enumerate}
\item[(i)]  a frame  $\F_\varphi=e^{Q_\varphi/2}\F_{e}$ has an $A$-tight dual frame;
\item[(ii)] the operator $Q_\varphi+(\ln{A})I$ is nonnegative in $\K$ and 
$\dim\mathcal{R}(A-e^{-Q_\varphi})\leq{e[\F_\varphi]}.$
\end{enumerate}
\end{cor}
\begin{proof}
$(i)\to(ii)$.
The frame operator coincides with multiplication by $A$ for $A$-tight frames.
Hence, each $A$-tight frame has the form 
$\F_\psi=\sqrt{A}\F_{e^o}$, where $\F_{e^o}$ is a PF in $\K$.
On the other hand, dual frames of $\F_\varphi$ are determined  by \eqref{KK23}.
Comparing these formulas and taking into account that  $e^{Q_\psi/2}$ and ${{\F}}_{Ue'}$ are determined uniquely by  $\F_\psi$ we arrive at the conclusion that
$$
\sqrt{A}=e^{Q_\psi/2}=\sqrt{e^{-Q_\varphi/2}e^Qe^{-Q_\varphi/2}}, \qquad \F_{e^o}={{\F}}_{Ue'}.
$$ 
The first relation means that $e^Q=Ae^{Q_\varphi}$. Therefore, $Q=Q_\varphi+(\ln{A})I$.
Due to Theorem \ref{thm5}, the operator $Q$ is nonnegative and
$$
\dim\mathcal{R}(I-e^{-Q})=\dim\mathcal{R}(I-\frac{1}{A}e^{-Q_\varphi})=\dim\mathcal{R}(A-e^{-Q_\varphi})\leq{e[\F_\varphi]}.  
$$

$(ii)\to(i)$. By the assumption, $Q=Q_\varphi+(\ln{A})I$ is nonnegative and 
\eqref{NNN961} holds. Moreover $Q$ commutes with $Q_\varphi$. Using 
\eqref{KKK1} we obtain a dual frame ${\F}_\psi=\sqrt{A}{\F}_{e'}$ of $\F_\varphi$ that is $A$-tight.
\end{proof}

\begin{rem}\label{NEWNEW}
Frames hawing Parseval duals were investigated in detail in \cite{BB, DH} by other methods. 
Corollary \ref{cor1} is equivalent to \cite[Proposition 2.4]{DH} for $A=1$.   
\end{rem}
 
\section{Dual frames for near-Riesz bases}\label{Sec2.4} 
 A frame with finite excess is called a near-Riesz basis. A near-Riesz basis 
 contains a Riesz basis as a subset and it behaves in many respects like Riesz basis
 \cite{Holub}. 
 
Let $\F_\varphi$ be a near-Riesz basis. Its
duals are described by the general formula \eqref{NNN960} in Theorem \ref{thm5}. 
This formula can be improved  by specifying the complementary PF
$\F_m$ in the Naimark dilation theorem (see Theorem \ref{Naimark})
for the case of near-Riesz bases. The corresponding version of
of Theorem \ref{Naimark} is proved below.

\subsection{Naimark dilation theorem for near-Riesz bases}\label{sec3.1}
Let a PF $\F_e=\{e_j,  j\in\mathbb{J}\}$ be a near-Riesz basis. 
Then  the index set can be decomposed $\mathbb{J}=\mathbb{J}_0\cup\mathbb{J}_1$ in such a way
 that $\F_e^0=\{e_j, j\in\mathbb{J}_0\}$ is a Riesz basis in $\K$ and $e[\F_e]=|\mathbb{J}_1|$. 

By virtue of Proposition \ref{WW1}, there exist a self-adjoint operator $Q_0$ and 
an orthonormal basis $\F_b=\{b_j, j\in\mathbb{J}_0\}$ in $\K$ such that $\F_e^0=e^{{Q_0}/2}\F_b$.

Denote 
\begin{equation}\label{AGHWMS44}
\M_1=\mbox{span}\{e_j,   j\in\mathbb{J}_1\},  \qquad  \F_e^1=\{e_j, j\in\mathbb{J}_1\}.
\end{equation}
By the construction, $\dim\M_1\leq|\mathbb{J}_1|={e[\F_e]}<\infty$ and $\F_e^1$ is a frame in $\M_1$.
\begin{lemma}\label{AGHWMS55}
The subspace $\M_1$ coincides with $\mathcal{R}(I-e^{{Q_0}})$ and the frame operator $S$ of $\F_e^1$ coincides with the restriction of $I-e^{{Q_0}}$ onto $\M_1$, i.e. $S=(I-e^{{Q_0}})|_{\M_{1}}$. 
\end{lemma}
 \begin{proof} 
 The frame operator of $\F_e^1$ is $S=I-e^{{Q_0}}$ because, for all  $g\in\M_1$,
 $$
 {S}g=\sum_{j\in\mathbb{J}_1}(g, e_j)e_j=g-\sum_{j\in\mathbb{J}_0}(g, e_j)e_j=g-\sum_{j\in\mathbb{J}_0}(e^{{Q_0}/2}g, b_j)e^{{Q_0}/2}b_j=(I-e^{{Q_0}})g.
 $$
 This means that $\M_1=\mathcal{R}((I-e^{{Q_0}})|_{\M_{1}})$. Assume that
 $f\in\K\ominus\M_1$. Then  
 $$
 f=\sum_{j\in\mathbb{J}}(f, e_j)e_j=\sum_{j\in\mathbb{J}_0}(f, e_j)e_j=\sum_{j\in\mathbb{J}_0}(f, e^{{Q_0}/2}b_j)e^{{Q_0}/2}b_j=e^{Q_0}f
 $$
 and, hence, $f\in\ker(I-e^{{Q_0}})$.  Thus $\K\ominus\M_1\subset\ker(I-e^{{Q_0}})$ and
 $\M_1=\mathcal{R}((I-e^{{Q_0}})|_{\M_{1}})=\mathcal{R}(I-e^{{Q_0}})$.
 \end{proof}

Since 
$$
\|f\|^2=\sum_{j\in\mathbb{J}_0}|(f, e^{{Q_0}/2}b_j)|^2+\sum_{j\in\mathbb{J}_1}|(f, e_j)|^2=\|e^{{Q_0}/2}f\|^2+\sum_{j\in\mathbb{J}_1}|(f, e_j)|^2, \quad f\in\K,
$$
the operator $e^{{Q_0}/2}$ is a contraction  in $\K$. 
Hence, $e^{{Q_0}/2}\leq{I}$ and $e^{-{Q_0}/2}\geq{I}$. This means that 
the operator $(e^{{-Q_0}}-I)^{1/2}$ is well-defined in $\K$. Moreover, 
in view of Lemma \ref{AGHWMS55}, its restriction onto $\M_1$ is invertible.
Therefore, the vectors 
 \begin{equation}\label{KKK114}
 \{k_j^0=e_j\oplus(e^{{-Q_0}}-I)^{1/2}e_j \}_{j\in\mathbb{J}_0}, \qquad \{k_j^1=e_j\oplus-(e^{{-Q_0}}-I)^{-1/2}e_j\}_{j\in\mathbb{J}_1}
 \end{equation}
 are well-defined and they belong to the Hilbert space $\K_{ext}=\K\oplus\M_1$. 
 
 Denote by $\mathcal{L}_0$ and $\mathcal{L}_1$ the subspace of  $\K_{ext}$ 
 generated by $\{k_j^0\}_{j\in\mathbb{J}_0}$ and $\{k_j^1\}_{j\in\mathbb{J}_1}$,
 respectively.
\begin{lemma}\label{GRR1}
The set $\{k_j^0\}_{j\in\mathbb{J}_0}$ is an orthonormal basis of $\mathcal{L}_0$, while
$\{k_j^1\}_{j\in\mathbb{J}_1}$ is a PF in $\mathcal{L}_1$ with the excess $e[\F_e]-\dim\M_1$.
The decomposition 
$\K_{ext}=\mathcal{L}_0\oplus\mathcal{L}_1$ holds.
\end{lemma}
\begin{proof}
The set $\{k_j^0\}_{j\in\mathbb{J}_0}$ is an orthonormal system in $\K_{ext}$ since
 $$
 (k_j^0, k_i^0)=(e_j, e_i)+((e^{{-Q_0}}-I)e_j, e_i)=(e^{{-Q_0}}e_j, e_i)=(b_j, b_i)=\sigma_{ji}.
 $$
  By the definition
 of $\mathcal{L}_0$, the set $\{k_j^0\}_{j\in\mathbb{J}_0}$ is an orthonormal basis of $\mathcal{L}_0$.

The subspaces $\mathcal{L}_0$ and $\mathcal{L}_1$ are orthogonal in $\K_{ext}$ since
$$
(k_j^0, k_i^1)=(e_j, e_i)+((e^{{-Q_0}}-I)^{1/2}e_j, -(e^{{-Q_0}}-I)^{-1/2}e_i)=(e_j, e_i)-(e_j, e_i)=0.
$$

Let a vector $f\oplus{g}\in\K_{ext}$ be orthogonal to $\mathcal{L}_0$. Then, for every $k_j^0$,
$$
0=(k_j^0, f\oplus{g})=(e_j, f)+((e^{{-Q_0}}-I)^{1/2}e_j, g)=(e_j, f+(e^{{-Q_0}}-I)^{1/2}g).
$$
Therefore, $f=-(e^{{-Q_0}}-I)^{1/2}g$ (since the set $\{e_j=e^{Q_0/2}b_j\}_{j\in\mathbb{J}_0}$
is complete in $\K$) and, as a result, 
$$
\K_{ext}\ominus\mathcal{L}_0=\{r=-(e^{{-Q_0}}-I)^{1/2}g\oplus{g}, \ g\in\M_1\}\supset{\mathcal{L}_1}.
$$ 
This means that $\dim\K_{ext}\ominus\mathcal{L}_0=\dim\M_1<\infty$.
On the other hand, by the construction, $\dim\mathcal{L}_1=\dim\M_1$. Therefore, 
 $\K_{ext}=\mathcal{L}_0\oplus\mathcal{L}_1$ and ${\mathcal{L}_1}=\{r=-(e^{{-Q_0}}-I)^{1/2}g\oplus{g}, \ g\in\M_1\}$.

To complete the proof one should verify that $\{k_j^1\}_{j\in\mathbb{J}_1}$ is a PF in $\mathcal{L}_1$.
For each vector $r=-(e^{{-Q_0}}-I)^{1/2}g\oplus{g}$ from $\mathcal{L}_1$, we get 
$$
|(r, k_j^1)|=|([(e^{{-Q_0}}-I)^{1/2}+(e^{{-Q_0}}-I)^{-1/2}]g, e_j)|=|(e^{-Q_0/2}g, (I-e^{Q_0})^{-1/2}e_j)|.
$$
In view of  Lemma \ref{AGHWMS55}, $S^{-1/2}\F_e^1=\{(I-e^{Q_0})^{-1/2}e_j, j\in{\mathbb{J}_1}\}$ is a PF in $\M_1$.
 Therefore, for all $r=-(e^{{-Q_0}}-I)^{1/2}g\oplus{g}\in\mathcal{L}_1$,
 $$
 \sum_{j\in\mathbb{J}_1}|(r, k_j^1)|^2=\sum_{j\in\mathbb{J}_1}|(e^{-Q_0/2}g, S^{-1/2}e_j)|^2=\|e^{-Q_0/2}g\|^2=\|r\|^2.
 $$
 The obtained relation means that $\{k_j^1\}_{j\in\mathbb{J}_1}$ is a PF in $\mathcal{L}_1$. Its excess is
 $|\mathbb{J}_1|-\dim\M_1=e[\F_e]-\dim\M_1$.
\end{proof} 
 
 Now we are ready to prove the Naimark dilation theorem for near-Riesz bases.
\begin{thm}\label{Naimark2}
If  a PF $\F_e=\{e_j,  j\in\mathbb{J}\}$ is  a near-Riesz basis, then there exists a Hilbert space $\M_2$ and a complementary PF
$\F_{m}=\{m_j, \ j\in \mathbb{J}\}$ in $\M_2$ such that $\F_{h}=\{{h}_j, \ j\in\mathbb{J}\}$, where
$$
h_j=\left\{\begin{array}{l}
e_j\oplus(e^{-{Q_0}}-I)^{1/2}e_j\oplus{0}, \quad  j\in \mathbb{J}_0 \vspace{2mm}\\
e_j\oplus-(e^{-{Q_0}}-I)^{-1/2}e_j\oplus{m_j}, \quad  j\in\mathbb{J}_1
\end{array}\right.
$$
is an orthonormal basis of $\Hil=\K\oplus\M_1\oplus\M_2$. 
The dimension of $\M_2$ coincides with $e[\F_e]-\dim\M_1$, where the space $\M_1$ is defined in \eqref{AGHWMS44}
\end{thm}
\begin{proof}
 By Lemma \ref{GRR1},  the set $\{k_j^1, j\in\mathbb{J}_1\}$ 
 is a PF in $\mathcal{L}_1$ with the excess $e[\F_e]-\dim\M_1$. For this PF, in view of  Theorem \ref{Naimark}, there exists a complementary PF
$\F_{m}=\{m_j, \ j\in \mathbb{J}_1\}$ in a Hilbert space  $\M_2$ with $\dim\M_2=e[\F_e]-\dim\M_1$  and  such that $\{k_j^1\oplus{m_j, j\in\mathbb{J}_1}\}$ is an orthonormal basis in $\mathcal{L}_1\oplus\M_2$.

Since $\{k_j^0, j\in\mathbb{J}_0\}$ is an orthonormal basis in $\mathcal{L}_0$
we arrive at the conclusion that the vectors
$$
h_j=\left\{\begin{array}{l}
{k_j^0}\ {\oplus} \ {0} \ {\oplus} \ {0}, \quad  j\in \mathbb{J}_0 \vspace{2mm}\\
0\ {\oplus}\ k_j^1\ {\oplus}{m_j}, \quad  j\in\mathbb{J}_1
\end{array}\right.
$$
form an orthonormal basis in the Hilbert space $\mathcal{L}_0{\oplus}\mathcal{L}_1{\oplus}{\M_2}$.  
To complete the proof it suffices: (a) to note that 
$\mathcal{L}_0{\oplus}\mathcal{L}_1{\oplus{\M_2}}=\K{\oplus}\M_1{\oplus}{\M_2}$; 
(b) to consider the extended PF $\F_{m}=\{m_j, \ j\in \mathbb{J}\}$ in $\M_2$ assuming $m_j=0$ when $j\in\mathbb{J}_0$;
(c) to use \eqref{KKK114}.
\end{proof}
 
\subsection{Dual frames for near-Riesz bases}\label{sec3.2}
Let a frame $\F_\varphi$ be a near-Riesz basis in $\K$.
Dual frames for $\F_\varphi$ are described in Theorem \ref{thm5}.
In the present section we make these results more precise
taking into account the specific of near-Riesz bases. 
Namely, if $\F_\varphi=e^{Q_\varphi/2}\F_e$ is a
near-Riesz basis, then the PF $\F_e$ is a near-Riesz basis too 
and $e[\F_\varphi]=e[\F_e]<\infty$ (Lemma \ref{WW1b}). Combining
the Naimark dilation Theorem \ref{Naimark2} for the near-Riesz basis $\F_e$ with 
general results of Theorem \ref{thm5} we obtain
\begin{cor}
Given a near-Riesz basis $\F_\varphi=e^{Q_\varphi/2}\F_{e}$ in $\K$.  Each dual frame $\F_\psi=\{\psi_j, j\in\mathbb{J}\}$ of $\F_\varphi$ consists of the vectors
\begin{equation}\label{AGHWMS101b}
\psi_j=\left\{\begin{array}{l}
e^{-Q_\varphi/2}[I+e^{Q/2}W_{12}(e^{-{Q_0}}-I)^{1/2}]e_j, \quad  j\in\mathbb{J}_0  \vspace{2mm} \\
e^{-Q_\varphi/2}[I-e^{Q/2}W_{12}(e^{-{Q_0}}-I)^{-1/2}]e_j+e^{Q/2}W_{12}m_j, \quad  j\in\mathbb{J}_1
\end{array}\right.,
\end{equation}
where $Q$ is a bounded nonnegative operator in 
$\K$  such that $\dim\mathcal{R}(I-e^{-Q})\leq{e[\F_\varphi]}$.
Furthermore, $\{m_j\}$ 
is the complementary PF of $\F_e$ in $\M_2$ (see Theorem \ref{Naimark2})
and  $W_{12}:\M\to\K$ is a part of an unitary operator $W$ defined by \eqref{DDD1}
and acting in $\K\oplus\M$, where $\M=\M_1\oplus\M_2$.
\end{cor}

The canonical dual frame $\F_\psi=e^{-Q_\varphi/2}\F_e$ is distinguished by $Q=0$ in \eqref{AGHWMS101b}   (in this case $W_{12}=0$, see  Corollary \ref{DDD2b}).

Another interesting alternative dual frame can be obtained from  \eqref{AGHWMS101b} if we set
$Q=-Q_0$. Such a choice is possible because $-Q_0$ is nonnegative (since $e^{Q_0/2}$ is a contraction in $\K$, see the proof of Lemma \ref{AGHWMS55}) and $\dim\mathcal{R}(I-e^{-Q})=\dim\mathcal{R}(I-e^{Q_0})=\dim\M_1\leq{e[\F_e]}={e[\F_\varphi]}$.
Choosing $Q=-Q_0$ and assuming that $W$ is defined by \eqref{eq:ww} (in this case $W_{12}|_{\M_1}= (I-e^{-Q})^{1/2}$ and $W_{12}|_{\M_2}=0$)
we obtain possibly simplest alternate dual frame of $\F_\varphi$
$$
\F_\psi=\{\psi_j=e^{-Q_\varphi/2}e^{-Q_0}e_j=e^{-Q_\varphi/2}e^{-Q_0/2}b_j, \ j\in\mathbb{J}_0\}\cup\{\psi_j=0, \  j\in\mathbb{J}_1\},
$$
where $\{e^{-Q_0/2}b_j, \ j\in\mathbb{J}_0\}$ is the bi-orthogonal Riesz basis for $\F_e^0=e^{{Q_0}/2}\F_b$.

\section*{Acknowledgements}
This work was partially supported by the Faculty of Applied Mathematics AGH UST statutory tasks within subsidy of 
Ministry of Science and Higher Education.

\end{document}